\documentclass[11pt]{amsart}
\usepackage{graphicx} 

\usepackage{amsmath}
\usepackage{amsthm, amsfonts, amssymb}

\usepackage{comment}
\usepackage{bbm}
\usepackage[left=3cm,right=3cm,top=3cm,bottom=3cm]{geometry}
\usepackage{pgfmath}

\usepackage[dvipsnames]{xcolor}
\usepackage{tikz}
\usepackage{tikz-3dplot}
\usetikzlibrary{patterns, calc}

\usetikzlibrary{3d, decorations.pathmorphing, decorations.pathreplacing, calc}

\theoremstyle{definition}

\newtheorem{theorem}{Theorem}[section]
\newtheorem*{theorem*}{Theorem}
\newtheorem{prop}[theorem]{Proposition}

\newtheorem{lemma}[theorem]{Lemma}
\newtheorem*{lemma*}{Lemma}
\newtheorem*{acknowledgment}{Acknowledgment}

\theoremstyle{definition}

\newtheorem{example}[theorem]{Example}

\theoremstyle{remark}
\newtheorem{remark}[theorem]{Remark}

\newcommand\blfootnote[1]{%
	\begingroup
	\renewcommand\thefootnote{}\footnote{#1}%
	\addtocounter{footnote}{-1}%
	\endgroup
}

\title{An inverse and a stability result for Ruzsa's inequality on triple sumsets}
\author{Swaroop Hegde}

\begin{document}

\begin{abstract}
	Ruzsa's inequality states that \(|A+A+A| \leq |A+A|^{3/2}\) for any finite set \(A\) in a commutative group. Ruzsa has constructed examples showing that this inequality is sharp asymptotically, up to a constant factor. We prove an inverse result which says that if \(|A+A+A| \geq \frac{1}{M} |A+A|^{3/2}\) for some parameter \(M,\) then the set \(A\) resembles the sets in Ruzsa's construction. We then construct more families of examples which suggest that our inverse result is likely best possible qualitatively. The method extends to give an inverse result for a higher sumset analogue of Ruzsa's inequality, namely \(|(h+1)A| \leq |hA|^{\frac{h+1}{h}}\) for any \(h\geq 2.\) We also provide a ``99\%-stability" version of Ruzsa's inequality, which describes near optimal structures when \(M\) is very close to \(1.\)
\end{abstract}

\maketitle \blfootnote{Department of Mathematics, University of Georgia, Athens, Georgia 30602, United States. Email: \texttt{swaroop.hegde@uga.edu}. This material is based upon work supported by the National Science Foundation under Grant No. 2054214.} 

\section{Introduction}

If \(A\) is a subset of a commutative group then its sumset is defined as follows:
\[A+A = \{a+b  : a, b \in A\}.\]

If \(A\) is finite and non-empty, then size of \(A+A\) can be trivially bounded by \begin{equation} \label{eq:A+A_trivial}
	|A| \leq |A+A| \leq \binom{|A|+1}{2}.
\end{equation}

The lower bound in Equation \eqref{eq:A+A_trivial} is attained when for instance \(A\) is a coset of a subgroup and the upper bound is attained when all the sums of \(A\) are distinct, such as in a geometric progression inside  \(\mathbb Z\) or when \(A\) is the set of generators of a free group. The quantity \[K = \frac{|A+A|}{|A|}\] called the doubling ratio is generally used to measure how big \(|A+A|\) is compared to \(|A|.\) The general philosopy in additive combinatorics is that when the doubling ratio is small, then the set \(A\) has high additive structure and resembles a coset of a subgroup. This is, for example, the essence of Freiman's theorem \cite{freiman_foundations} and there have been recent breakthroughs in the quantitative aspects of this theorem \cite{Marton_GGMF}. Such results which give information about the additve structure of a set just by knowing something about the size of its sumset fall under the bigger umbrella of inverse theorems in combinatorics. 
 Classical questions in additive combinatorics also concern cardinality estimates of iterated sumsets of a set \(A\) given the size of \(|A+A|.\) (The iterated sumset \(hA\) for \(h\geq 2\) consists of all the \(h\)-fold sums of \(A\).) For instance, Pl\"unnecke's inequality \cite{Plunnecke_og, ruzsa_graph}, a standard tool in additive combinatorics, says that if \(|A+A| = K|A|\) then for \(h\geq 3\) \begin{equation}\label{eq:Plunnecke}
	|hA|\leq K^{h}|A|.
\end{equation} So if \(K\) is not large, then the sizes of iterated sumsets do not change by very much, like in cosets.

Note however, that if \(K\geq |A|^{1-\frac 1 h}\) then Pl\"unnecke's inequality is worse than the trivial bound \(|hA| \leq |A|^h\). Ruzsa investigated this and first proved a bound  in the case \(h=3\) which complements Pl\"unnecke's inequality when \(K\) is large (see Theorem 1.9.1 of \cite{ruzsa_sumsets_and_structure}), namely \begin{equation} \label{eq:Ruzsa_ineq_original}
	|A+A+A| \leq |A+A|^{3/2} = K^{3/2}|A|^{3/2}.
\end{equation}

We will be refering to Equation \eqref{eq:Ruzsa_ineq_original} as Ruzsa's inequality in this paper. Pl\"{u}nnecke's inequality gives the better estimate when $K \leq |A|^{1/3}$ and Ruzsa's inequality is better when $|A|^{1/3} \leq K \leq |A|.$ Ruzsa's inequality is asymptotically sharp up to a constant factor as is shown by the following example.

\begin{example}[Additively dissociated sets] \label{ex:dissociated_set} Let us call a subset \(A\) of a commutative group \textit{\(h\)-dissociated} for some integer \(h\geq 2\) if all the \(h\)-fold sums of \(A\) are distinct. (Note that if a set is \(h\)-dissociated, then it is also \((h-1)\)-dissociated.) For example, $A \subset \mathbb Z$ given by the geometric progression $A = \{ 1, 3, 3^2, \ldots, 3^{n-1} \}$ is \(3\)-dissociated. We have $|A| = n$ with $$ |A+A| = \binom{n+1}{2} \qquad \text{and} \qquad |A+A+A| = \binom{n+2}{3}. $$
	
	Thus we have \( |A+A+A| = \Omega( |A+A|^{3/2})\) and \(A\) serves as an asymptotically sharp example for Ruzsa's inequality.
\end{example}

 	However, in this case the doubling ratio is \( K= \frac{|A+A|}{|A|} = \frac{|A|+1}{2},\) which is the largest possible, whereas Ruzsa's inequality is better than Pl\"unnecke's inequality for a wider range, namely \(|A|^{1/3} \leq K \leq |A|.\) Ruzsa has constructed the following family of examples (Theorem 1.9.5 of \cite{ruzsa_sumsets_and_structure}) showing that Ruzsa's inequality is sharp asymptotically, upto a constant factor, for any possible doubling in the range $|A|^{1/3} \leq K \leq |A|.$ 
 
  \begin{example} [Ruzsa's construction] \label{ex:ruzsa_construction}
 Let $m \geq 1$ be an integer and $K \in \mathbb R$ be such that $m \leq K \leq m^3.$ We construct $ A \subset \mathbb{Z}^3$ as the union of two sets $A = X \cup Y,$ where
 
 $$X = \{ (x, y, z) \in \mathbb Z^3 \ | \ 0\leq x, y, z < m \} \ \text{and}$$ $$ Y = \{ (x, 0, 0), (0, x, 0), (0, 0, x) \in \mathbb Z^3 \ | \ 0\leq x < \sqrt{Km^{3}} \}.$$

 Visually, $A$ consists of a cube in $\mathbb Z^3$ cornered at the origin with sides of length $m,$ together with three splines, each of length $\sqrt{Km^{3}},$ attached to it along the coordinate axes.
 
 \begin{center}

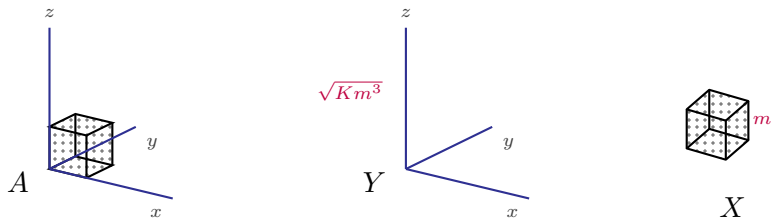
\begin{tabular}{p{0.25\textwidth} p{0.3\textwidth} p{0.2\textwidth}}
      \tdplotsetmaincoords{70}{35} 
\begin{tikzpicture}[tdplot_main_coords, thick]

  \def\L{.6} 

  \coordinate (O) at (0,0,0);
  \coordinate (A) at (\L,0,0);
  \coordinate (B) at (\L,\L,0);
  \coordinate (C) at (0,\L,0);
  \coordinate (D) at (0,0,\L);
  \coordinate (E) at (\L,0,\L);
  \coordinate (F) at (\L,\L,\L);
  \coordinate (G) at (0,\L,\L);

  \fill[pattern=dots, pattern color=gray] (O) -- (A) -- (B) -- (C) -- cycle; 
  \fill[pattern=dots, pattern color=gray] (A) -- (B) -- (F) -- (E) -- cycle; 
  \fill[pattern=dots, pattern color=gray] (C) -- (B) -- (F) -- (G) -- cycle; 
  \fill[pattern=dots, pattern color=gray] (O) -- (C) -- (G) -- (D) -- cycle; 
  \fill[pattern=dots, pattern color=gray] (O) -- (A) -- (E) -- (D) -- cycle; 
  \fill[pattern=dots, pattern color=gray] (D) -- (E) -- (F) -- (G) -- cycle; 

  \draw[thick] (O) -- (A) -- (B) -- (C) -- cycle;  
  \draw[thick] (D) -- (E) -- (F) -- (G) -- cycle;  
  \draw[thick] (O) -- (D);
  \draw[thick] (A) -- (E);
  \draw[thick] (B) -- (F);
  \draw[thick] (C) -- (G);

  \draw[-, thick, Blue] (0,0,0) -- (2,0,0) node[anchor=north east]{{\tiny\color{darkgray}$x$}};
  \draw[-, thick, Blue] (0,0,0) -- (0,2,0) node[anchor=north west]{{\tiny\color{darkgray}$y$}};
  \draw[-, thick, Blue] (0,0,0) -- (0,0,2) node[anchor=south]{{\tiny\color{darkgray}$z$}};

 \node at (-0.1, -0.7, 0) {{ $A$}};
    
\end{tikzpicture}& \tdplotsetmaincoords{70}{35} 
\begin{tikzpicture}[tdplot_main_coords, thick]

  \draw[-, thick, Blue] (0,0,0) -- (2,0,0) node[anchor=north east]{{\tiny\color{darkgray}$x$}};
  \draw[-, thick, Blue] (0,0,0) -- (0,2,0) node[anchor=north west]{{\tiny\color{darkgray}$y$}};
  \draw[-, thick, Blue] (0,0,0) -- (0,0,2) node[anchor=south]{{\tiny\color{darkgray}$z$}};
    \node at (0, -1.3, 1.5) {{\tiny \color{purple}{$\sqrt{Km^{3}}$}}};
  \node at (-0.1, -0.7, 0) {{ $Y$}};

\end{tikzpicture} &  \tdplotsetmaincoords{60}{120} 
\begin{tikzpicture}[tdplot_main_coords, line join=round]

  \def\L{.6} 

  \coordinate (O) at (0,0,0);
  \coordinate (A) at (\L,0,0);
  \coordinate (B) at (\L,\L,0);
  \coordinate (C) at (0,\L,0);
  \coordinate (D) at (0,0,\L);
  \coordinate (E) at (\L,0,\L);
  \coordinate (F) at (\L,\L,\L);
  \coordinate (G) at (0,\L,\L);

  \fill[pattern=dots, pattern color=gray] (O) -- (A) -- (B) -- (C) -- cycle; 
  \fill[pattern=dots, pattern color=gray] (A) -- (B) -- (F) -- (E) -- cycle; 
  \fill[pattern=dots, pattern color=gray] (C) -- (B) -- (F) -- (G) -- cycle; 
  \fill[pattern=dots, pattern color=gray] (O) -- (C) -- (G) -- (D) -- cycle; 
  \fill[pattern=dots, pattern color=gray] (O) -- (A) -- (E) -- (D) -- cycle; 
  \fill[pattern=dots, pattern color=gray] (D) -- (E) -- (F) -- (G) -- cycle; 

  \draw[thick] (O) -- (A) -- (B) -- (C) -- cycle;  
  \draw[thick] (D) -- (E) -- (F) -- (G) -- cycle;  
  \draw[thick] (O) -- (D);
  \draw[thick] (A) -- (E);
  \draw[thick] (B) -- (F);
  \draw[thick] (C) -- (G);

    \node[purple] at (0.5, 1.1, 0.7) {{\tiny $m$}};

    \node[black] at (1.8, 1.4, 0.1) {$X$};
\end{tikzpicture}

  \end{tabular}
  
\end{center}
 
 Then it can be checked that \(|A| = \Theta(m^3)\), \(|A+A| = \Theta(K|A|)\) and \(|A+A+A| = \Theta (K^{3/2}|A|^{3/2}).\) (We omit the details here but similar calculations can be found in Examples \ref{ex:gap}, \ref{ex:random} and \ref{ex:ruzsa_higher_sums}.)
\end{example}

The main result of this paper is an inverse result showing that if \(|A+A+A| \geq \frac{1}{M} K^{3/2}|A|^{3/2},\) for some parameter \(M,\) then \(A\) is similar to the sets in Ruzsa's construction. More specifically, we show the existence of a subset \(Y\subset A\) of size \(|Y| = \Theta_M(\sqrt{K|A|})\) such that \(|Y+Y+Y| = \Theta_M(|Y|^3)\) (with the dependence on \(M\) being polynomial in \(M\)), which is analogous to the splines in Ruzsa's construction.
\begin{theorem}[Inverse result] \label{thm:inverse_theorem_2}
	Let $A$ be a finite non-empty subset of a commutative group with $|A+A| \leq K|A|.$ Suppose $ |A+A+A| \geq \frac{1}{M}K^{3/2}|A|^{3/2},$ where $1\leq  M \leq \frac{1}{2}\sqrt{\frac{|A|}{K}}.$ Then there exists $ Y \subset A$ such that \begin{enumerate}
		\item $ \displaystyle \frac{\sqrt{K|A|}}{(2M)^{1/3}} \leq |Y| < 2M{\sqrt{K|A|}}$ and 
		\item $ \displaystyle |Y+Y+Y| \geq \frac{|Y|^3}{(2M)^4}.$
	\end{enumerate}
\end{theorem}

Furthermore, in Examples \ref{ex:gap} and \ref{ex:random} of this paper we modify Ruzsa's construction to obtain more extremal examples to show that one could not hope to say much structurally about the set \(A\setminus Y\) which is analogous to the cube in Example \ref{ex:ruzsa_construction}. The key feature that these examples have in common is ``having splines", making our inverse result qualitatively best possible. The conclusions of Theorem \ref{thm:inverse_theorem_2} are nontrivial even when \(M\) is \(|A|^{\varepsilon}\) for a small \(\varepsilon > 0.\)

In Section \ref{sec:perspectives}, we will survey various approaches that have appeared in the literature to prove Ruzsa's inequality and provide new perspectives. 
In Theorem \ref{thm:macaulay_sumsets}, we will state Ruzsa's inequality in a slightly sharper form (with respect to the multiplicative constant), while characterizing the extremal cases exactly to be \(3\)-dissociated sets described in Example \ref{ex:dissociated_set}. This is a known result by the work of Eliahou and Mazumdar \cite{eliahou_mazumdar}, but we provide a new and short proof here. Using these ideas we obtain a stability version of Ruzsa's inequality. 

\begin{theorem}[Stability for Theorem \ref{thm:macaulay_sumsets}] \label{thm:structural_stability_ruzsa}
	Let \(\delta <1/1152\) and \(A\) be a large enough\footnote{This will be made clear later in the proof in Section \ref{sec:stability}.} subset of a commutative group. Suppose \( |A+A| = \binom{x+1}{2}\) for some real number \( x\geq 1.\) If \[|A+A+A| \geq (1-\delta)\binom{x+2}{3},\] then there exists a subset \(Y\subset A\) with \(\displaystyle (1-6\delta^{1/2})x < |Y| < (1+9\delta^{1/2})(x+1)\) such that \[ |Y+Y+Y| \geq (1-50\delta^{1/2}) \binom{|Y| + 2}{3}.\] 
\end{theorem}

The above result says that if \(|A+A+A|\) is very nearly optimal, then \(A\) is close to Example \ref{ex:dissociated_set}, in the sense that \(A\) contains a subset \(Y\) which is almost \(3\)-dissociated and \(Y+Y+Y\) contains most of \(A+A+A.\) Such results which describe nearly optimal structures have come to be known as ``99\%-stability" results in combinatorics. We will note here that our inverse result Theorem \ref{thm:inverse_theorem_2} gives structural information about \(A\) even if \(|A+A+A|\geq \frac{1}{M}K^{3/2}|A |^{3/2}\) where \(M\) could be large, in contrast to Theorem \ref{thm:structural_stability_ruzsa}. 

Generalizations of Ruzsa's inequality in different directions can be found in the works of Gyarmati-Matolcsi-Ruzsa \cite{DifferentSumGyaMaRu, GyMaRu}, Balister-Bollob\'as \cite{balister_bollobas} and Madiman-Marcus-Tetali \cite{madiman_tetali_marcus}. The most natural one is an estimate of the size of \((h+1)A\) in terms of the size of \(hA\), namely
\begin{equation}\label{eq:Ruzsa_hvs(h-1)}
	|(h+1)A| \leq |hA|^{\frac{h+1}{h}} = \alpha^{\frac{h+1}{h}}|A|^{\frac{h+1}{h}},
\end{equation}
where \(\alpha = \frac{|hA|}{|A|}.\) Ruzsa's construction (Example \ref{ex:ruzsa_construction}) extends in a straightforward way to show that inequality \eqref{eq:Ruzsa_hvs(h-1)} is sharp up to a constant factor (depending on \(h\)). In our discussion upto Section \ref{sec:inverse_result} of this paper we will restrict to the original case (\(h=2\)) for ease of exposition. However, much of it can be easily extended to a general \(h\geq 2\) and this is described in Section \ref{sec:general_inverse result}.

\begin{acknowledgment}
	The author would like to thank Cosmin Pohoata for suggesting the problem and for helpful conversations, as well as Noga Alon for valuable inputs regarding the Clements-Lindstr\"om Theorem (Proposition \ref{prop:clements_lindstrom}). The author is grateful to Giorgis Petridis for his constant encouragement and many helpful discussions.
\end{acknowledgment}

\section{Perspectives} \label{sec:perspectives}

A common theme in the various proofs of Ruzsa's inequality found in the literature is to use the isoperimetric nature of estimating the size of \(A+A+A\) in terms of the size of \(A+A.\) In this section, we will first survey some of these different perspectives and in Theorem \ref{thm:macaulay_sumsets} we will restate Ruzsa's inequality in a slightly sharper form in the sense of a better multiplicative constant. We would like to note at the beginning itself that this result has been obtained in different forms previously in the works of Noga Alon \cite{Alon_results_1} and Eliahou-Mazumdar \cite{eliahou_mazumdar}, but our goal here is to tie together the different perspectives. 

\textbf{Loomis-Whitney Inequality:} The discrete Loomis-Whitney inequality (in dimension 3) \cite{loomis_whitney} states that if $S \subset X_1\times X_2 \times X_3$, where $X_1, X_2$ and \(X_3\) are arbitrary sets and $S$ is finite, then $$|S|^2 \leq |\pi_{12} (S) | \cdot |\pi_{23}(S)| \cdot |\pi_{13}(S)|,$$ where \(|\cdot|\) denotes cardinality and \( \pi_{12} (S) \) is the projection of \(S\) into \(X_1\times X_2\) and so on. There are many well-known generalizations of the Loomis-Whitney inequality including the Uniform Cover Inequality \cite{balister_bollobas} and Shearer's Inequality \cite{shearer} from information theory. 
A proof of Ruzsa's inequality using the Loomis-Whitney inequality appears in the paper of Gyarmati-Matolcsi-Ruzsa \cite{GyMaRu}, a proof using Shearer's lemma appears in the paper of Madiman-Marcus-Tetali \cite{madiman_tetali_marcus} and a combination of both appears in the paper
of Balister-Bollob\'as \cite{balister_bollobas}.

We will briefly illustrate here the common ideas that they all use. Given a commutative group $G$ and a finite subset $A\subset G$, let `$\leq$' be any total order on the elements of A. This can be extended to $A\times A$ and $A\times A\times A$ via the corresponding lexicographic orders. We embed $A+A+A$ in $A\times A\times A$ by mapping a sum $s\in A+A+A$ to the lexicographically minimal triple $(x, y, z) \in A\times A\times A$ such that $x + y + z = s.$ Let us call this embedding $S,$ so that $|S| = |A+A+A|.$ In a similar fashion, we embed $A+A$ in $A\times A$ by taking the lexicographically minimal representative for each sum. If we call this embedding $C,$ then we have $|C| = |A+A|.$ Now we make the following simple but crucial observation. 

\begin{lemma}[Key observation]\label{lem:projection_observation}
	Let \(S\subset A\times A\times A\) and \(C\subset A\times A\) be as defined above (the lexicographically minimal embeddings of \(A+A+A\) and \(A+A\) respectively). Then we have \(\pi_{12}(S)\subset C, \pi_{13}(S)\subset C\) and \(\pi_{23}(S)\subset C.\)
\end{lemma}

\begin{proof}
	Let \((x, y, z) \in S\) and suppose that there exists $(x', y')\leq(x, y)$ in $A\times A$ such that $x + y = x' + y'.$ Then it follows that $(x', y', z) \leq (x, y, z)$ and $x' + y' + z = x + y + z.$ So by the minimality of $(x, y, z)$ in $A \times A \times A,$ this would imply $x = x'$ and $y = y'.$ Similarly, it follows that $(y, z) \in C$ and \((x, z) \in C.\) 
\end{proof}
 
Thus we have \(|\pi_{12} (S) | \leq |A+A|, |\pi_{23} (S) | \leq |A+A|\) and \(|\pi_{13} (S) | \leq |A+A|.\) Then a simple application of the Loomis-Whitney inequality to the set $S$ yields Ruzsa's inequality. 

The above argument treats the co-ordinates assymetrically, whereas due to commutativity of addition they ought to be symmetric. In fact, following the above argument, let us define a new set $\tilde S \subset A\times A\times A$ by including $S$ and all permutations of elements of $S$. The size of $\tilde S$ is about 6 times the size of $S,$ minus some lower order terms while the sizes of 2-projections of $\tilde S$ are at most 2 times the sizes of respective 2-projections of $S.$ By applying the Loomis-Whitney inequality to this modified set $\tilde S,$ one can asymptotically improve the multipicative constant in Ruzsa's inequality to get 

\begin{equation}\label{eq:ruzsa_sqrt2/3}
	|A+A+A| \leq \left( \frac{\sqrt 2}{3} + o(1) \right) |A+A|^{3/2}.  
\end{equation}

\textbf{Kruskal-Katona Theorem:} Equation \eqref{eq:ruzsa_sqrt2/3} can also be proved by using the Kruskal-Katona theorem instead of the Loomis-Whitney inequality. The Kruskal-Katona theorem concerns families of sets and their shadows, but we will work with the equivalent (hyper)graph version.  We will use here a condensed verison of the Kruskal-Katona theorem due to Lov\'asz \cite{lovasz_problems} which states that if the number of edges in a graph $G$ is given by $|G| = \binom{x}{2}$ for some real number $x\geq 2,$ then the number of triangles in the graph is at most $\binom{x}{3}$. We can also rewrite this to say that the number of triangles in a graph is at most \(\frac{\sqrt{2}}{3} |E(G)|^{3/2}.\) 

Moving back to sumsets, let $S\subset A\times A\times A$ and $C \subset A\times A$ be the sets as defined before Lemma \ref{lem:projection_observation}. Let us also define the sets
$$ S' = \{ (x, y, z) \in S \ | \ x \neq y \ \text{and} \ y\neq z\} \quad \text{and}\quad C' = \{ (x, y) \in C \ | \ x \neq y \}.$$ Consider a graph $G$ with vertex set $A$ such that $\{x, y\} \in E(G)$ iff $(x, y) \in C'$ or $(y, x)\in C'.$ Using Lemma \ref{lem:projection_observation} we see that if $(x, y, z) \in S'$ then \(\{x, y\}, \{y, z\}\) and \(\{x, z\}\) are all edges in \(G,\) and thus $x, y, z$ form a triangle in $G.$ So by the Kruskal-Katona Theorem, we obtain \begin{align*}
	|S'| \leq \frac{\sqrt{2}}{3} |C'|^{3/2}\leq \frac{\sqrt{2}}{3} |C|^{3/2}.
\end{align*}

Note that $|S'| \geq |S| - 2|C|$ by Lemma \ref{lem:projection_observation}, because $S'$ is obtained from $S$ by removing elements of the form $(x, x, y)$ and $(x, y, y).$ Thus we get 
\begin{align}\label{eq:KK_Ruzsa}
	|S| \leq \frac{\sqrt{2}}{3} |C|^{3/2} + 2|C|,
\end{align}

which yields \eqref{eq:ruzsa_sqrt2/3}.

After writing a first draft of the paper, the author learned that a similar sumset analogue of the Kruskal-Katona theorem was also obtained independently by Noga Alon in \cite{Alon_results_1} (Theorem 2.3), in the restricted setting where the group \(G\) consists only of idempotent elements, though the argument holds in general.

\textbf{Clements-Lindstr\"om Theorem:} In the Kruskal-Katona argument above, the lower order term \(2|A+A|\) seems a little wasteful. The reason we got it is that we could not handle repeated terms in a sum using the Kruskal-Katona Theorem, since the theorem only applies to sets and not multisets. One can however get around this issue by using a generalization of the Kruskal-Katona theorem to multisets due to Clements and Lindstr\"om \cite{clements-lindstrom}. We will do this in Theorem \ref{thm:macaulay_sumsets}. The specific case of the Clements-Lindstr\"om theorem we will use was first proved by Macaulay \cite{macaulay}, but in a different setting of Hilbert functions of certain modules. We refer the reader to an article by Greene and Kleitman \cite{greene-kleitman} (see Section 8) which presents all of these results in a unified combinatorial setting.

In order to state the Clements-Lindstr\"om theorem let us set up some notation. Let \(X = \{x_1, x_2, \ldots\}\) be a ground set with a total ordering \(x_1 < x_2 < \ldots\) and let \(X^{(k)}\) denote the set of all \(k\)-element multi-sets with elements from \(X.\) The (lower) shadow of a set \( B \subset X^{(k)} \) is defined as follows \[ \partial B = \{ R \in X^{(k-1)} \ : \ R \subset P \text{ for some } P \in B \}. \] We define the lexicographic order on \(X^{(k)}\) as follows: if \(P, Q \in X^{(k)}\) are distinct then \( P \leq Q  \) if the smallest element of \(P \Delta Q\) (the symmetric difference) belongs to \(P.\)

Recall also the fact that when \(\binom{x+1}{2}\) is considered as a polynomial in the real variable $x,$ then it is an increasing function and that any real number $r\geq 1$ can be represented uniquely as \(r = \binom{x+1}{2}\) for some real number \(x \geq 1\).

\begin{prop}[Clements-Lindstr\"om \cite{clements-lindstrom}, Macaulay \cite{macaulay}] \label{prop:clements_lindstrom}
	 If \(S \subset X^{(3)}\) is finite and \(|\partial S| = \binom{x+1}{2}\) for some real number \(x \geq 1,\) then \( |S| \leq \binom{x+2}{3}. \) Equality holds if and only if \(x\) is an integer and \(S\) is the set of all \(3\)-element multi-subsets of an \(x\)-element set.
\end{prop}

The following result is the analogue of Clements-Lindstr\"om theorem for sumsets and a slight strengthening of Ruzsa's inequality. It was first obtained by Eliahou and Mazumdar in \cite{eliahou_mazumdar}. They translate the problem to the set up of Hilbert functions of standard graded algebras and then make use of the algebraic version of Macaulay's theorem. The argument bears some resemblance to Khovanskii's Theorem \cite{khovanskii}. Our proof is more direct and connects to previous approaches within additive combinatorics. We shall also note that it is possible to state a slightly sharper version of Proposition \ref{prop:clements_lindstrom} (and hence Theorem \ref{thm:macaulay_sumsets}, as is given in \cite{eliahou_mazumdar}) if we write \(|\partial S|\) as a sum of binomial coefficients (see \cite{greene-kleitman}), but we will stick to the condensed version, similar to the condensed version of Kruskal-Katona theorem due to Lov\'asz \cite{lovasz_problems}. See also \cite{eliahou_mazumdar_optimal}, where sharp examples are constructed for the stronger form, in the semigroup setting.

\begin{theorem}[Eliahou-Mazumdar \cite{eliahou_mazumdar}] \label{thm:macaulay_sumsets}
	 Let \( A \) be a finite subset of a commutative group. If \(|A+A| = \binom{x+1}{2}\) for some \(x\geq 1\), then \[|A+A+A| \leq \binom{x+2}{3}.\] This is best possible if \(x\) is an integer and in this case equality holds if and only if \(A\) is 3-dissociated.
\end{theorem}
\begin{proof}
	Let \(S \subset A^{(3)}\) be the set of lexicographically minimal members of \(A^{(3)}\) which represent sums in \(A+A+A\) and let $C$ be the set of lexicographically minimal members of \(A^{(2)}\) which represent sums in \(A+A. \) Our observation in Lemma \ref{lem:projection_observation} still holds in this new setting, that each multiset in the lower shadow of an element of \(S\) is the lexicographically minimal representative in \(A^{(2)}\) of a sum in \(A+A,\) and thus \(\partial S \subseteq C.\) In fact, one can also show that \( C \subseteq \partial S,\) by checking that if \( \{x_1, x_2\} \in C \) then \( \{a_1, x_1, x_2\} \in S,\) where \(a_1\) is the smallest element of \(A\) in the total order on \(A.\) So we have \( |\partial S| = |A+A| = \binom{x+1}{2} \) and the Clements-Lindstr\"om theorem implies that \(|A+A+A| = |S| \leq \binom{x+2}{3}\). 
	
	When \(x\) is an integer we see that equality can hold when \(A\) is \(3\)-dissociated as in Example \ref{ex:dissociated_set}. Moreover, the characterization of extremal cases in the Clements-Lindstr\"om Theorem (Proposition \ref{prop:clements_lindstrom}) shows that equality holds in this case only when \(A\) is a 3-dissociated set. 
	\end{proof}


\section{Stability using Kruskal-Katona} \label{sec:stability}
 
Keevash has proved the following stability version of Kruskal-Katona theorem in \cite{keevash_kk}. Recall that equality is attained in the Kruskal-Katona Theorem when the graph is complete. The stability version says that if the number of triangles in a graph is very close to being maximal, then there is a part of the graph which contains most of the triangles and is close to being complete.

\begin{theorem}[Keevash \cite{keevash_kk}] \label{thm:keevash_kk_stability}
	Suppose \(0 < \varepsilon < 1/2, x \geq 3(1+\varepsilon), \delta <\varepsilon^2/144\) and \(G\) is a graph with \(\binom{x}{2}\) edges. If the number of triangles in \(G\) is at least \((1-\delta)\binom{x}{3},\) then:
	\begin{enumerate}
		\item There exists a set \(Z\) of \(\lceil x \rceil\) vertices in \(G\) such that all but at most \(20(\varepsilon^{-1}+1)\delta^{1/2}\binom{x}{2}\) edges of \(G\) are contained in \(Z.\)
		\item There exists a set \(Y\) of size \(|Y| < (1 + 6\delta^{1/2})x\) vertices that contains at least \(\binom{(1-4\delta^{1/2})x}{3}\) triangles.
	\end{enumerate}
\end{theorem}

Our proof of Ruzsa's inequality using the Kruskal-Katona theorem can be combined with Theorem \ref{thm:keevash_kk_stability} to obtain a stability version of Ruzsa's inequality (as restated in Theorem \ref{thm:structural_stability_ruzsa}). However, we can only get stability if \(|A+A+A|\) is away from the upper bound by a factor of at most \((1-\delta)\), where \(\delta\) is bounded by a very small constant and \(A\) is large enough. We have not tried to be very careful about the dependence of \(|A|\) on \(\delta\) in the below theorem, since it is not the main result of the paper and only serves to provide contrast between different techniques. In the next section, we will provide an inverse result (Theorem \ref{thm:inverse_theorem_2}) which allows \(|A+A+A|\) to be much farther from the upper bound than in the following result.

We should also remark that the author first tried to use a stability result for Loomis-Whitney inequality by Ellis et. al. \cite{lw_stability} to obtain a structural result when \(|A+A+A| \geq (1-\delta)|A+A|^{3/2}.\) However we obtained a contradiction when \(\delta < \frac{1}{1133}\) and \(|A|\) is large, which make sense in hindsight because of the improved constant in Equation \eqref{eq:ruzsa_sqrt2/3} as opposed to Equation \eqref{eq:Ruzsa_ineq_original}. This is what motivated us in the first place to try and improve the multiplicative constant in Ruzsa's inequality and hence discover the other perspectives.

\begin{theorem*}[Theorem \ref{thm:structural_stability_ruzsa} restated] 
	Let \(\delta <1/1152\) and \(A\) be a large enough subset of a commutative group. Suppose \( |A+A| = \binom{x+1}{2}\) for some real number \( x\geq 1.\) If \[|A+A+A| \geq (1-\delta)\binom{x+2}{3},\] then there exists a subset \(Y\subset A\) with \(\displaystyle (1-6\delta^{1/2})x < |Y| < (1+9\delta^{1/2})(x+1)\) such that \[ |Y+Y+Y| \geq (1-50\delta^{1/2}) \binom{|Y| + 2}{3}.\] 
\end{theorem*}
	
	When we say \(A\) is large enough, we mean that \(|A|\) is large enough so that the lower order terms appearing in our calculations below are small enough to be absorbed in the main terms; this will be clarified in the proof. Note that \(x \rightarrow \infty \) as \(|A| \rightarrow \infty,\) since \(\binom{x+1}{2} = |A+A| \geq |A| \) and \(\binom{x+1}{2}\) is a monotonically increasing function of \(x.\)
	
	Theorem \ref{thm:structural_stability_ruzsa} tells us that if \(A\) is large enough and \(|A+A+A|\) is a very close to the upper bound in Theorem \ref{thm:macaulay_sumsets}, then the set \(A\) contains a set of size roughly \(x\) which is close to being additively dissociated, where \(x\) is as defined in Theorem \ref{thm:structural_stability_ruzsa}.

\begin{proof}[Proof of Theorem \ref{thm:structural_stability_ruzsa}]
	 Suppose \(|A+A+A| \geq (1-\delta)\binom{x+2}{3}\). 
	 
	 Define sets \(S \subset A\times A \times A, C \subset A\times A\) and the graph \(G\) as in the proof of Equation \eqref{eq:ruzsa_sqrt2/3} above using the Kruskal-Katona theorem. We know that \(|G| \leq |A+A| = \binom{x+1}{2}.\) Let \(T\) denote the set of triangles in \(G.\)  Then by the Kruskal-Katona theorem, we have \(|T| \leq \binom{x+1}{3}.\)
	 
	 The proof of Equation \eqref{eq:ruzsa_sqrt2/3} implies that \(|A+A+A| \leq |T| + 2|A+A|.\) So  
	 \begin{align*}
	 		|T| &\geq |A+A+A| - 2 |A+A| \\
	 			&\geq (1-\delta)\binom{x+2}{3} - 2 \binom{x+1}{2} \\
	 			&\geq (1-2\delta)\binom{x+1}{3}.
	 \end{align*}
	 
	 Above, we are using the assumption that \(|A|\) (and hence \(x\)) is large enough so that \((1+\delta) \binom{x+1}{2} \leq \delta\binom{x+1}{3}.\)
	 
	 Then by Theorem \ref{thm:keevash_kk_stability}, there exists a set \(Y\) of vertices such that \(|Y| \leq (1+6\sqrt2\delta^{1/2})(x+1)\) that contains at least \(\binom{(1-4\sqrt 2\delta^{1/2})(x+1)}{3}\) triangles. (Note that this also implies that \(|Y| \geq (1-6\delta^{1/2})x \).) Let \(T_Y\) denote the set of triangles contained inside \(Y\) and let \(T_Y^c = T\setminus T_Y.\) Then we have 
	 \begin{align}
	 	|T_Y^c| &= |T| - |T_Y| \notag \\
	 				  &\leq \binom{x+1}{3} - \binom{(1-4\sqrt 2\delta^{1/2})(x+1)}{3} \notag \\
	 				  &\leq 4\sqrt{2} \delta^{1/2} (x+1)\binom{x+1}{2} \notag \\
	 			      &\leq 12\sqrt{2}\delta^{1/2} \binom{x+2}{3}. \label{eq:triangle_Yc_<}
	 \end{align}
	 
	 In the second-to-last step we have used the estimate that if \(a>b>2\) then \(\binom{a}{3} - \binom{b}{3} \leq (a-b) \binom{a}{2}\) (see Lemma 7 of \cite{keevash_kk}), and in the final step we use once again the fact that \(x\) is large enough.
	 
	 On the other hand, we partition \(S\) into \(S_Y = S \cap (Y\times Y\times Y)\) and \(S_Y^c = S \setminus S_Y.\) By the definition of \(S\) as the lexicographically minimal representatives of \(A+A+A\) we have \(|A+A+A| = |S| = |S_Y| + |S_Y^c|.\) First we observe that there is an injection from \(S_Y\) into \(Y+Y+Y,\) since elements of \(S_Y\) represent distinct elements of \(Y+Y+Y\), and so \[|S_Y| \leq |Y+Y+Y|.\]  Now consider the set \(|S_Y^c|,\) which consists of triples in \(S\) with at least one element outside of \(Y.\) Recall from Lemma \ref{lem:projection_observation} that any 2-projection of a triple in \(S\) (and in particular, \(S_Y^c\)) lies in \(C.\) In particular, a member of \(S_Y^c\) which is a triple of distinct elements would correspond to a triangle in \(G\) with atleast one vertex outside of \(Y\), i.e., a triangle in \(T_Y^c.\) On the other hand, the number of triples in \(S_Y^c\) with repeated elements can be crudely bounded above by \(2|A+A|\) as we did earlier, and so \(|S_Y^c| \leq |T_Y^c| + 2|A+A|.\) 
	 
	 Thus, we have 
	 \begin{align*}
	 	|A+A+A| = |S_Y| + |S_Y^c| \leq |Y+Y+Y| + |T_Y^c| + 2|A+A|.
	 \end{align*}
	 
	 We can rewrite this equation using the hypothesis on $|A+A+A|$ and Equation \eqref{eq:triangle_Yc_<} as 
	 \begin{align*}
	 	|Y+Y+Y| &\geq |A+A+A| - |T_Y^c| - 2|A+A| \\
	 				  &\geq (1-\delta)\binom{x+2}{3} -  12\sqrt{2} \delta^{1/2}\binom{x+2}{3}  - 2\binom{x+1}{2}  \\
	 				  &\geq (1-18\delta^{1/2}) \binom{x+2}{3}-2\binom{x+1}{2}  \\
	 				  &\geq(1-20\delta^{1/2})\binom{x+2}{3}. 
	 \end{align*}
	 
	  The second-to-last step follows from the fact that \(\delta < \delta^{1/2}\) and that \(12\sqrt2 + 1 < 18\), and in the final step we use once again the fact that \(x\) is large enough.
	  
	  The above equation says that \(Y+Y+Y\) contains most of \(A+A+A.\) If \(x\) is large enough, then using the fact that \(|Y| \leq (1+6\sqrt2 \delta^{1/2})(x+1)\) we can rewrite the above equation as \[|Y+Y+Y| \geq (1-50\delta^{1/2})\binom{|Y|+2}{3}.\] Above, we have used another estimate on binomial co-efficients, namely if \(a>b>0,\) then \(\binom{a}{3}/\binom{b}{3} > (a/b)^3\) (Fact (2) of \cite{keevash_kk}).
\end{proof}
 
\section{An inverse result} \label{sec:inverse_result}

In this section we will prove an inverse result to Ruzsa's inequality (Equation \eqref{eq:Ruzsa_ineq_original}), which says that if $|A+A|\leq K|A|$ and $|A+A+A|$ is close to the order of $K^{3/2}|A|^{3/2},$ then there is a subset of $A$ which behaves similar to the set of three splines in Ruzsa's construction (Example \ref{ex:ruzsa_construction}). Recall that the nature of the set of splines $Y$ in Ruzsa's examples is that $|Y+Y|$ is of the order $|Y|^2$ and $|Y+Y+Y|$ is of the order $|Y|^3.$ Moreover, the size of the set Y is of the order $\sqrt{K|A|}.$ 


\begin{theorem*}[Inverse result restated]
	Let $A$ be a finite non-empty subset of a commutative group with $|A+A| \leq K|A|.$ Suppose $ |A+A+A| \geq \frac{1}{M}K^{3/2}|A|^{3/2},$ where $1\leq  M \leq \frac{1}{2}\sqrt{\frac{|A|}{K}}.$ Then there exists $ Y \subset A$ such that \begin{enumerate}
		\item $ \displaystyle \frac{\sqrt{K|A|}}{(2M)^{1/3}} \leq |Y| < 2M{\sqrt{K|A|}}$ and 
		\item $ \displaystyle |Y+Y+Y| \geq \frac{|Y|^3}{(2M)^4}.$
	\end{enumerate}
\end{theorem*}

The proof of our inverse result bears some resemblance to the analysis found in a proof of Ruzsa's inequality (in a more general setting) by Gyarmati-Matolcsi-Ruzsa (Theorem 4.4 in \cite{GyMaRu}). Let us first recall what is referred to as Pl\"unnecke's inequality for a large subset (Theorem 1.7.2 in \cite{ruzsa_sumsets_and_structure}), which gives us control over the size of the subset which shows restricted growth in Pl\"unnecke's inequality, but at the cost of losing in the constant in the upper bound. We will state here a version with an improved constant due to Petridis \cite{petridis_highersums}. A more general version can be found in \cite{DifferentSumGyaMaRu}.


\begin{theorem}[Pl\"unnecke's inequality for a large subset] \label{thm:Plunnecke_for_large_hv2}
	Let $A$ be a finite non-empty subset of a commutative group such that $|A+A| \leq K|A|$ and let $0<\delta<1$ be a parameter. Then there exists $X \subseteq A$ of size  $|X| \geq (1-\delta)|A|$ such that for all \(h \geq 2\) $$ |X+hA| \leq \frac{K^h}{\delta^{h-1}}|A|. $$
\end{theorem}
We will now prove the inverse result.
\begin{proof}[Proof of Theorem \ref{thm:inverse_theorem_2}]
    Suppose $ |A+A+A| \geq \frac{1}{M}K^{3/2}|A|^{3/2},$ where \(M\) is as stated in the theorem.

    By applying Pl\"unnecke's inequality for a large subset (with \(h=2\)) to $A$ we obtain $X \subset A$ such that $|X| \geq (1-\delta)|A|$ and  \begin{align*}
        |X+A+A| \leq \frac{K^2}{\delta}|X|,
    \end{align*}
    where \(\delta\) is a parameter we will chose later.

    We set $Y= A\setminus X$ and note that $A+A+A \subset (X+A+A) \cup (Y+Y+Y).$ So we have
    \begin{align}
        |A+A+A| &\leq |X+A+A| + |Y+Y+Y| \notag \\
        &\leq \frac{K^2}{\delta}|A| + |Y+Y+Y|  \label{eq:}
    \end{align}

    Now, by using the above inequality and the hypothesis that $ |A+A+A| \geq \frac{1}{M}K^{3/2}|A|^{3/2}$ we obtain \begin{align*}
        |Y+Y+Y| \geq \frac{1}{M} K^{3/2}|A|^{3/2} - \frac{K^2}{\delta}|A|.
    \end{align*}

    We now choose $\delta = 2M\sqrt{\frac{K}{|A|}}$ so that  \begin{align}
        |Y+Y+Y| \geq \frac{1}{2M}K^{3/2}|A|^{3/2} \label{eq:large_tripling_Y}.
    \end{align}
     Note that the hypotheses guarantee that $0< \delta < 1$ and moreover the choice of $\delta$ guarantees that $|Y+Y+Y| \geq \frac{1}{2}|A+A+A|.$
     Furthermore, $\displaystyle |Y| \leq \delta|A| = 2M\sqrt{K|A|}.$ Thus, inequality \eqref{eq:large_tripling_Y} tells us that that \begin{align*}
        |Y+Y+Y| \geq \frac{|Y|^3}{2^4M^4}.
    \end{align*}

    On the other hand, by using inequality \eqref{eq:large_tripling_Y} and the trivial bound $|Y+Y+Y| \leq |Y|^3$ we get $$ |Y| \geq \frac{\sqrt{K|A|}}{(2M)^{1/3}}.$$
\end{proof}

\section{More Extremal Examples}

Given Theorem \ref{thm:inverse_theorem_2}, one might ask if it is possible to say anything about the set $X = A \setminus Y,$ that appeared in the above proof and is analogous to the slow-growing subset (cube) in Ruzsa's examples (Example \ref{ex:ruzsa_construction}). In Ruzsa's construction, \(X\) is the cartesian product of an interval in \(\mathbb Z\). In Example \ref{ex:random} below, we replace \(X\) by a dense random subset of the cube, implying that $X$ need not have much structure. To make calculations easier we will work in $\left(\mathbb Z /n\mathbb Z\right)^3$ for a suitable $n.$

\begin{example}\label{ex:random}
	We let \(m \in \mathbb Z^+, p\gg \frac{1}{m}\) and \(K \in \mathbb R^+\) be parameters such that \(m/p \leq K\leq pm^3.\) By making \(m\) large enough we can find these parameters such that \(n = p^{1/2}K^{1/2}m^{3/2}\) is an integer and \(m\) divides \(n.\) Consider the group \(G = (\mathbb Z / n \mathbb Z)^3\) and let \(H\) be the subgroup of \(G\) isomorphic to \((\mathbb Z / m \mathbb Z)^3.\) We now define \(A\subset G\) as \(A:=X \cup Y\) where \(X\) is a randomly chosen subset of \(H\) such that \(\mathbb P (x \in X) = p\) for all \(x\in H\) independently, and \[Y = (\mathbb Z/n\mathbb Z\times \{0\} \times \{0\})\cup( \{0\} \times \mathbb Z/n\mathbb Z \times \{0\})\cup(\{0\} \times \{0\}\times \mathbb Z/n\mathbb Z).\]
	
	Note that \(|X| = \Theta(pm^3)\) with high probability and \(|Y| = \Theta(p^{1/2} K^{1/2}m^{3/2}).\) By the assumption that \(K\leq pm^3,\) with high probability \(|X|\) has bigger order than \(|Y|.\) So \(|A| = \Theta (pm^3)\) with high probability.
	
	Next we check that \(|A+A| = \Theta( K|A|)\) with high probability. First we have  \[|A+A| \leq |X+X| + |X+Y| + |Y+Y|.\] 
	\(|X+X| \leq |H+H| = m^3\) since \(X\) is a subset of the subgroup \(H\), and \(|Y+Y| = (3+o(1))n^2 = (3+o(1)) pKm^3.\) Note that \(pK\geq m\) by our assumption that \(m/p\leq K,\) which means that \(|Y+Y|\) is of bigger order than \(|X+X|.\)
	
	In order to calculate \(|X+ Y|\) let us first think about what happens when we add a single point in \(X\) to the spline along \(z\)-direction in \(Y\). Since each spline has length \(n\), we obtain a line parallel to it and of length \(n\). This tells us that the size of the sum of \(X\) with the \(z\)-spline is \(|\pi_{xy}(X)|\times n\), where \(\pi_{xy}(X)\) is the projection of \(X\) onto the \(xy\)-plane. A similar calculation holds for the sum of \(X\) with the \(x\) and \(y\) splines respectively. Recall that \(|\pi_{xy}(X)|= \Theta(m^2)\)
	with high probabilty. Indeed, for any \((x,y) \in [m]^2\) we have \[\mathbb P ((x, y) \not \in \pi_{xy}(X)) = (1-p)^m.\]
	So by a union bound
	\[\mathbb P ((x, y) \not \in \pi_{xy}(X) \text{for some \((x,y) \in [m]^2\)}) \leq m^2 (1-p)^m \leq m^2 e^{-pm}.\]
	
	So long as \(p \gg 1/m \) (which we did assume initially) this probability is exponentially small. Thus \(|X+ Y|= O(m^2n) = O(m^3p^{1/2}K^{1/2}m^{1/2}) \leq O(pKm^3)\) with high probability, where in the last step we have used once again the assumption that \(m\leq Kp\).
	
	Putting all this together we have with high probability 
	\[ |A+ A| \leq |X+ X| + |X+ Y| + |Y + Y|= O(pKm^3). \]
	
	Since \(|A+ A|\geq |Y + Y|= \Omega(pKm^3)\), it follows that \(|A+ A| = \Theta(pKm^3) = \Theta(K|A|)\) with high probability. 
	
	Finally we have \(|A+ A+ A|\geq |Y + Y + Y|= n^3.\) So \(|A+ A+ A|= \Theta (K^{3/2}|A|^{3/2})\) with
	high probability.
\end{example}

The sets \(X\) in Examples \ref{ex:ruzsa_construction} and \ref{ex:random} are different, yet both have the property that \(|X+X+X| = O(|X+X|).\) It is natural to ask if such a property is common in all nearly extremal examples of Ruzsa's inequality. The following example gives a negative answer. 

\begin{example} \label{ex:gap}
	In Ruzsa's examples (Example \ref{ex:ruzsa_construction}), the set $X$ (a cube) was the cartesian product of the interval $[0, m)$ in \(\mathbb Z.\) We will replace this with the cartesian product of a generalized arithmetic progression as follows.
	
	
	Fix \(k , d \in \mathbb Z^+\) and let \(P\) denote the set of integers in base \(3k\) with digits less than $k.$ More formally,
	\[ P:= \left\{ x_0+x_1(3k) + x_2(3k)^2 + \cdots + x_{k-1}(3k)^{d-1} \ : \ x_i\in [0, k) \text{ for each } 0\leq i \leq d-1 \right\}. \]
	
	Note that $|P| = k^d$ (so this is a proper generalized arithmetic progression of dimension $d$) and $P\subset[m].$ It is easy to see that $|P+P|$ is the set of integers in base \(3k\) with digits less than $2k-1$ and so $|P+P| = (2k-1)^d.$ Similarly, \(|P+P+P|\) is the set of integers in base \(3k\) with digits less than $3k-2$ and $|P+P+P| = (3k-2)^d.$ 
	
	Let us define $X = P\times P\times P.$ Then $|X| = k^{3d}$ and since sumsets behave nicely with cartesian products, we have \( |X+X|=  (2k-1)^{3d} \) and \( |X+X+X|=  (3k-2)^{3d}. \)
	
	We then define $Y$ to be the set of splines with modified lengths as follows. Let $K$ be such that $k^d\leq K\leq k^{3d}$ and
	
	\[ Y = \left\{ (x, 0, 0), (0, x, 0), (0, 0, x) \in \mathbb Z^3 \ | \ 0\leq x < K^{\frac 1 2}k^{\frac{3d}{2}} \right\}. \] 
	
	We take \(A = X \cup Y\) and note that $|A| = \Theta(k^{3d}),$ since $|X|$ dominates over $|Y|$ because of the assumption that $K\leq k^{3d}.$
	
	Next we check that $|A+A| = \Theta(K|A|).$ In the below calculation, in order to find the size of the sum of a spline with $X,$ we note that this sumset is covered by a``parallelepiped" which extends $X$ along the spline.
	\begin{align*}
		|A+A| &\leq |X+X| + |X+Y| + |Y+Y| \\
		&= |X+X| + O(|P|^2|S|) + |Y+Y| \\
		&= (2k-1)^{3d} + O(k^{2d}\cdot K^{1/2}k^{3d/2}) + O(K k^{3d}) \\
		&= O(k^{3d}) + O(k^{3d}\cdot \underbrace{K^{1/2}k^{d/2}}_{k^d \leq K}) + O(K k^{3d}) \\
		&= O(K k^{3d}) = O(K|A|).
	\end{align*}
	
	Since \(|A+A| \geq |Y+Y| = \Omega(K |A|),\) we also have $|A+A| = \Theta(K|A|).$
	
	As for the triple sums, we have 
	\begin{equation*}
		|A + A + A| \geq |Y+Y+Y| = \Omega( K^{3/2}|A|^{3/2}).
	\end{equation*}
	
	So the set $A$ serves as an asymptotically sharp example for Ruzsa's inequality.
	
	The important thing to note here about the set $X$ is that \(\frac{|X+X|}{|X|} = (1+o(1))2^{3d}\) and $\frac{|X+X+X|}{|X+X|} = (1+o(1))\left(\frac{3}{2}\right)^{3d}.$ So by making the dimension \(d\) larger we can make these ratios as large as we want. 
	
\end{example}

\section{Higher sumsets} \label{sec:general_inverse result}
We will now state a generalization of the inverse result to higher sumsets. Given an integer \(h \geq 2,\) define \(\alpha=\frac{|hA|}{|A|}.\)  Recall the version of Ruzsa's inequality for higher sumsets in Equation \eqref{eq:Ruzsa_hvs(h-1)} \[|(h+1)A| \leq |hA|^{\frac{h+1}{h}} \leq \alpha^{\frac{h+1}{h}}|A|^{\frac{h+1}{h}}.\]

The following example is an extension of Ruzsa's construction (Example \ref{ex:ruzsa_construction}) for higher sumsets and shows that the above inequality is sharp up to constant factors (fixing \(h\)). Once again, to make calculations simpler, we will work in a group \(G = (\mathbb Z / n \mathbb Z)^{h+1}\) for a suitable \(n.\)
\begin{example} \label{ex:ruzsa_higher_sums}
	Fix an integer \(h\geq 2\)	and let \(m\in \mathbb Z^+\) and  \(\alpha \in \mathbb R^+\) be parameters such that \(m^{h-1} \leq \alpha \leq m^{h^2-1}.\) If \(m\) is large enough then we can find these parameters such that \(n = (\alpha m^{h+1})^{\frac{1}{h}}\) is an integer and \(m\) divides \(n.\) Let \(G\) be the group \(G = (\mathbb Z / n \mathbb Z)^{h+1}\) and let \(X\) be its subgroup isomorphic to \((\mathbb Z / m \mathbb Z)^{h+1}.\) For each \(1\leq i \leq h+1\) let \(Y_i\) denote the copy of \(\mathbb Z/n\mathbb Z\) inside \(G\) along the \(i\)-th coordinate direction and let \(Y = \cup_{i=1}^{h+1} Y_i\). We then define the set \(A\subset G\) as \(A:=X \cup Y.\)
		
		Note that $|A| = \Theta_h(m^{h+1})$ since $|X|$ has bigger order than $|Y|,$ which follows from the assumption that $\alpha \leq m^{h^2-1}.$
		
		Next we check that $|hA| = \Theta_h(\alpha|A|).$ In the calculations below, first we are using the fact that \(X\) is a subgroup, so \(lX = X\) for all \(l\geq 2.\) Then in equation \eqref{eq:sum X+(h-1)Y} below, in order to find the size of the sum of \((h-1)Y\) with $X$  we note that \((h-1)Y\) is a union of (so to speak) co-dimension 2 hyperplanes in \((\mathbb Z / n\mathbb Z)^{h+1}.\) So the sumset of \(X\) along one of these hyperplanes is covered by the union of translates of the cube \(X\) along the hyperplane. 
		\begin{align}
			|hA| &\leq |X + (h-1)A| + |hY| \notag \\
			&= O_h\left(\sum_{l=0}^{h-1}|X+lY|\right) + O_h(n^h)  \notag\\
			&= O_h\left(|X+(h-1)Y|\right) + O_h(\alpha m^{h+1})  \notag \\
			&= O_h(m^{2}\cdot \alpha^{\frac{h-1}{h}}m^{\frac{h^2-1}{h}}) + O_h(\alpha m^{h+1}) \label{eq:sum X+(h-1)Y} \\
			&= O_h(\alpha m^{h+1} \underbrace{\alpha^{-\frac{1}{h}}m^{\frac{h-1}{h}}}_{\leq 1, \text{ since } m^{h-1} \leq \alpha})+ O_h(\alpha m^{h+1})  \notag \\
			&= O_h(\alpha m^{h+1}) = O_h(\alpha|A|). \notag
		\end{align}
		
		Clearly we also have $|hA| \geq |hY| \geq \alpha m^{h+1} = \Omega_h(\alpha|A|).$ Finally we note \begin{equation*}
			|(h+1)A| \geq |(h+1)Y| = n^{h+1} = (\alpha |A|)^{\frac{h+1}{h}}.
		\end{equation*}
		This shows that Ruzsa's inequality \eqref{eq:Ruzsa_hvs(h-1)} is sharp up to a constant factor (fixing \(h\)), for \(\alpha\) in the range $|A|^{1-\frac{2}{h+1}} \leq \alpha \leq |A|^{h-1}.$

\end{example}

The following is an inverse result to Equation \eqref{eq:Ruzsa_hvs(h-1)}.
\begin{theorem}[Inverse result for \(h \geq 2\)] \label{thm:general_inverse_theorem}
	Let $A$ be a finite non-empty subset of a commutative group and \(h\geq 2\) be a fixed integer. Suppose $|hA| \leq \alpha|A|$ and that $ |(h+1)A| \geq \frac{1}{M}\alpha^{\frac{h+1}{h}}|A|^{\frac{h+1}{h}},$ where $1\leq  M \leq \frac{1}{2h}\left(\frac{|A|}{\alpha^{\frac{1}{h-1}}}\right)^{\frac{1}{h}}.$ Then there exists $ Y \subset A$ such that \begin{enumerate}
		\item $ \displaystyle \frac{(\alpha|A|)^{\frac{1}{h}}}{(2M)^{\frac{1}{h}}} \leq |Y| < (2Mh)^{(h-1)}(\alpha|A|)^{\frac{1}{h}}$ and 
		\item $ \displaystyle |(h+1)Y| \geq \frac{|Y|^{h+1}}{(2M)^{h^2}h^{h^2-1}}.$
	\end{enumerate}
\end{theorem}

The proof of Theorem \ref{thm:general_inverse_theorem} goes exactly the same way as that of Theorem \ref{thm:inverse_theorem_2} (\(h=2\) case) and uses following version of Pl\"unnecke's inequality for a large subset from Gyarmati-Matolcsi-Ruzsa \cite{DifferentSumGyaMaRu}. 

\begin{theorem}[Pl\"unnecke's inequality for a large subset \cite{DifferentSumGyaMaRu}] \label{thm:Plunnecke_large_general}
	Let $A$ be a finite non-empty subset of a commutative group and \(h\geq 2\) be an integer such that $|hA| \leq \alpha|A|.$ Let $0<\delta<1$ be a parameter. Then there exists $X \subseteq A$ of size  $|X| \geq (1-\delta)|A|$ such that $$ |X+hA| \leq \frac{h\alpha^{\frac{h}{h-1}}}{\delta^{\frac{1}{h-1}}}|A|. $$
\end{theorem}

\begin{remark}
	In Theorem \ref{thm:general_inverse_theorem}, the dependence of the constants on \(M\) is still polynomial in \(M,\) taking \(h\) to be fixed. However, the dependence on \(h\) is likely not optimal. If one wishes to involve the doubling ratio \(K = \frac{|A+A|}{|A|}\) instead of \(\alpha = \frac{|hA|}{|A}\) then we can iteratively apply Equation \eqref{eq:Ruzsa_hvs(h-1)} to obtain the estimate \[|(h+1)A| \leq K^{\frac{h+1}{2}}|A|^{\frac{h+1}{2}}.\]
	
	In this case, one could then use the same method along with Theorem \ref{thm:Plunnecke_for_large_hv2} to obtain the following result, which has a better dependence on \(h.\)
\end{remark}

\begin{theorem} \label{thm:inverse_theorem_hv2}
	Let $A$ be a finite non-empty subset of a commutative group with $|A+A| \leq K|A|.$ Suppose $ |(h+1)A| \geq \frac{1}{M}K^{\frac{h+1}{2}}|A|^{\frac{h+1}{2}},$ where $1\leq  M \leq \frac{1}{2}\left(\frac{|A|}{K}\right)^{\frac{h-1}{2}}.$ Then there exists $ Y \subset A$ such that \begin{enumerate}
		\item $ \displaystyle \frac{\sqrt{K|A|}}{(2M)^{\frac{1}{h+1}}} \leq |Y| < (2M)^{\frac{1}{h-1}}{\sqrt{K|A|}}$ and 
		\item $ \displaystyle |(h+1)Y| \geq \frac{|Y|^{h+1}}{(2M)^4}.$
	\end{enumerate}
\end{theorem}

We shall also note that much of what was said in Section \ref{sec:perspectives} of this paper can be generalized to higher sumsets as well. For example, using essentially the same argument as in Theorem \ref{thm:macaulay_sumsets} and a general version of the Clements-Lindstr\"om Theorem (see \cite{greene-kleitman}) we can prove the following.

\begin{theorem}[Eliahou-Mazumdar \cite{eliahou_mazumdar}] \label{thm:higher_macaulay_sumsets}
	Let \( A \) be a non-empty finite subset of a commutative group. If \(|hA| = \binom{x+h-1}{h}\) for some \(x\geq 1\), then \[|(h+1)A| \leq \binom{x+h}{h+1}.\] This is best possible if \(x\) is an integer and in this case, equality holds if and only if \(A\) is \((h+1)\)-dissociated (cf. Example \ref{ex:dissociated_set}).
\end{theorem}

One can also prove a stability result for \(h\geq 3\) similarly to Theorem \ref{thm:structural_stability_ruzsa} using the stability version of Kruskal-Katona theorem due to Keevash \cite{keevash_kk}, showing that if the size of \((h+1)A\) is close to the upper bound in Theorem \ref{thm:higher_macaulay_sumsets} then the set \(A\) contains an (almost) dissociated subset \(Y\) such that \((h+1)Y\) contains most of \((h+1)A.\)

\bibliographystyle{amsplain}
\bibliography{citations}

\end{document}